\documentclass[10pt]{article}
\usepackage{graphicx}
\usepackage{mathrsfs}
\usepackage{amsfonts}
\usepackage{amssymb}
\usepackage{amsmath}
\usepackage{amsthm}
\usepackage[bookmarks=true]{hyperref}

\DeclareMathOperator{\dist}{dist} \DeclareMathOperator{\CR}{CR}
\DeclareMathOperator{\Var}{Var}
\renewcommand{\Re}{\operatorname{Re}}

\newtheorem{theorem}{Theorem}[section]
\newtheorem{lemma}[theorem]{Lemma}
\newtheorem{proposition}[theorem]{Proposition}
\newtheorem{corollary}[theorem]{Corollary}
\theoremstyle{remark}
\newtheorem*{remark}{Remark}

\title{Limit theorems for critical first-passage percolation on the triangular lattice
}
\author{Chang-Long Yao \thanks{Academy of Mathematics and Systems Science,
CAS, Beijing, China (E-mail: deducemath@126.com) This work was
supported by the National Natural Science Foundation of China (No.
11601505) and the Key Laboratory of Random Complex Structures and
Data Science, CAS (No. 2008DP173182).}}

\begin{document}
\maketitle
\markboth{Sample paper for the {\protect\ntt\lowercase{amsmath}} package}
{Sample paper for the {\protect\ntt\lowercase{amsmath}} package}
\renewcommand{\sectionmark}[1]{}

\begin{abstract}
Consider (independent) first-passage percolation on the sites of the
triangular lattice $\mathbb{T}$ embedded in $\mathbb{C}$.  Denote
the passage time of the site $v$ in $\mathbb{T}$ by $t(v)$, and
assume that $P(t(v)=0)=P(t(v)=1)=1/2$.  Denote by $b_{0,n}$ the
passage time from 0 to the halfplane
$\{v\in\mathbb{T}:\mbox{Re}(v)\geq n\}$, and by $T(0,nu)$ the
passage time from 0 to the nearest site to $nu$, where $|u|=1$.  We
prove that as $n\rightarrow\infty$, $b_{0,n}/\log n\rightarrow
1/(2\sqrt{3}\pi)$ a.s., $E[b_{0,n}]/\log n\rightarrow
1/(2\sqrt{3}\pi)$ and Var$[b_{0,n}]/\log n\rightarrow
2/(3\sqrt{3}\pi)-1/(2\pi^2)$;  $T(0,nu)/\log n\rightarrow
1/(\sqrt{3}\pi)$ in probability but not a.s., $E[T(0,nu)]/\log
n\rightarrow 1/(\sqrt{3}\pi)$ and Var$[T(0,nu)]/\log n\rightarrow
4/(3\sqrt{3}\pi)-1/\pi^2$.  This answers a question of Kesten and
Zhang (1997) and improves our previous work (2014).  From this
result, we derive an explicit form of the central limit theorem for
$b_{0,n}$ and $T(0,nu)$.  A key ingredient for the proof is the
moment generating function of the conformal radii for conformal loop
ensemble CLE$_6$, given by Schramm, Sheffield and Wilson (2009).


\end{abstract}

\section{Introduction}

First-passage percolation (FPP) was introduced by Hammersley and
Welsh in 1965 as a model of fluid flow through a random medium.  We
refer the reader to the recent surveys \cite{26,6}.  In this paper,
we continue our study of critical FPP on the triangular lattice
$\mathbb{T}$, initiated in \cite{20}.  We focus on this particular
lattice because our proof relies on the existence of the scaling
limit of critical site percolation on $\mathbb{T}$ (see
\cite{2,23}), and this result has not been proved for other planar
percolation processes.  For recent progress on general planar
critical FPP, see \cite{21}.

Let $\mathbb{T}=(\mathbb{V},\mathbb{E})$ denote the triangular
lattice, where $\mathbb{V}:=\{x+ye^{\pi
i/3}\in\mathbb{C}:x,y\in\mathbb{Z}\}$ is the set of sites, and
$\mathbb{E}$ is the set of bonds, connecting adjacent sites.  Let
$\{t(v):v\in\mathbb{V}\}$ be an i.i.d. family of Bernoulli random
variables:
\begin{equation*}
P[t(v)=0]=P[t(v)=1]=\frac{1}{2}.
\end{equation*}
We call this model \textbf{Bernoulli critical FPP} on $\mathbb{T}$,
and denote by $P$ its probability measure.  Note that we can view
this model as critical site percolation on $\mathbb{T}$ (see e.g.
\cite{14,5} for background on two-dimensional critical percolation).
We usually represent it as a random coloring of the faces of the
dual hexagonal lattice $\mathbb{H}$, each face centered at $v\in
\mathbb{V}$ being blue ($t(v)=0$) or yellow ($t(v)=1$) with
probability $1/2$ independently of the others. Sometimes we view the
site $v$ as the hexagon in $\mathbb{H}$ centered at $v$.

A \textbf{path} is a sequence $v_0,\ldots,v_n$ of distinct sites of
$\mathbb{T}$ such that $v_{k-1}$ and $v_k$ are neighbors for all
$k=1,\ldots,n$.  For a path $\gamma$, we define its passage time as
$T(\gamma):=\sum_{v\in \gamma}t(v).$  The \textbf{first-passage
time} between two site sets $A,B$ is defined as
\begin{equation*}
T(A,B):=\inf \{T(\gamma):\gamma \mbox{ is a path from a site in $A$
to a site in $B$}\}.
\end{equation*}
For any $u\in\mathbb{C}$ with $|u|=1$, denote by $T(0,nu)$ the
first-passage time from 0 to the nearest site in $\mathbb{V}$ to
$nu$ (if there are more than one such sites, we choose a unique one
by some deterministic method). Denote by $b_{0,n}$ the first-passage
time from 0 to the halfplane $\{v\in \mathbb{V}:\Re(v)\geq n\}$.

Our main theorem below answers a question proposed by Kesten and
Zhang (see (1.10) and (1.11) in \cite{10}):

\begin{theorem}\label{t1}
\begin{align}
&\lim_{n\rightarrow\infty}\frac{b_{0,n}}{\log
n}=\frac{1}{2\sqrt{3}\pi}~~a.s.,\label{t14}\\
&\lim_{n\rightarrow\infty}\frac{E[b_{0,n}]}{\log
n}=\frac{1}{2\sqrt{3}\pi},\label{t12}\\
&\lim_{n\rightarrow\infty}\frac{\Var[b_{0,n}]}{\log
n}=\frac{2}{3\sqrt{3}\pi}-\frac{1}{2\pi^2}.\label{t22}
\end{align}
For each $u\in \mathbb{C}$ with $|u|=1$,
\begin{align}
&\lim_{n\rightarrow\infty}\frac{T(0,nu)}{\log
n}=\frac{1}{\sqrt{3}\pi}~~\mbox{in probability but not
a.s.,}\label{t13}\\
&\lim_{n\rightarrow\infty}\frac{E[T(0,nu)]}{\log
n}=\frac{1}{\sqrt{3}\pi},\label{t11}\\
&\lim_{n\rightarrow\infty}\frac{\Var[T(0,nu)]}{\log
n}=\frac{4}{3\sqrt{3}\pi}-\frac{1}{\pi^2}.\label{t21}
\end{align}
\end{theorem}
\begin{remark}
In \cite{20}, we proved the law of large numbers for $b_{0,n}$ and
$T(0,nu)$, but cannot give exact values of the limits.  The proof in
that paper relies on the subadditive ergodic theorem, which is a
nice tool to show the existence of the limit but gives no insight
for the exact value of the limit.
\end{remark}

Kesten and Zhang \cite{10} proved a central limit theorem for
$b_{0,n}$ and $T(0,nu)$.  Combining their result and Theorem
\ref{t1}, we obtain the explicit form of the CLT:

\begin{corollary}\label{c1}
\begin{equation*}
\frac{b_{0,n}-\frac{\log
n}{2\sqrt{3}\pi}}{\sqrt{\left(\frac{2}{3\sqrt{3}\pi}-\frac{1}{2\pi^2}\right)\log
n}}\stackrel{d}\longrightarrow N(0,1)\mbox{ as
$n\rightarrow\infty$}.
\end{equation*}
For each $u\in \mathbb{C}$ with $|u|=1$,
\begin{equation*}
\frac{T(0,nu)-\frac{\log
n}{\sqrt{3}\pi}}{\sqrt{\left(\frac{4}{3\sqrt{3}\pi}-\frac{1}{\pi^2}\right)\log
n}}\stackrel{d}\longrightarrow N(0,1)\mbox{ as
$n\rightarrow\infty$}.
\end{equation*}
\end{corollary}
\emph{Idea of the Proof.}  Instead of dealing with $b_{0,n}$ and
$T(0,nu)$ directly, we consider $c_n$ which is the first-passage
time from 0 to a circle of radius $n$ centered at 0.  We shall show
analogous limit theorem for $c_n$, then Theorem \ref{t1} follows
from this. Using a color switching trick we obtain that annulus time
has the same distribution as the number of cluster boundary loops
surrounding 0 in the annulus (under a monochromatic boundary
condition).  Camia and Newman's full scaling limit \cite{2} and
moment generating function of the conformal radii for CLE$_6$
\cite{15} allow us to derive limit theorem for the scaling limit of
annulus times.  Then from this we get the limit result for $c_n$ and
$E[c_n]$ easily.  In order to prove the limit result for
$\Var[c_n]$, we will use a martingale approach from \cite{10}.

\section{Notation and preliminaries}\label{s2}
We denote the underlying probability space by
$(\Omega,\mathscr{F},P)$, where $\Omega=\{\mbox{0,1}\}^{\mathbb{V}}$
(or $\{\mbox{blue,yellow}\}^{\mathbb{V}}$), $\mathscr{F}$ is the
cylinder $\sigma$-field and $P$ is the joint distribution of
$\{t(v):v\in \mathbb{V}\}$.  A \textbf{circuit} is a path whose
first and last sites are neighbors. For a circuit $\mathcal {C}$,
define
\begin{equation*}
\overline{\mathcal {C}}:= \mathcal {C}\cup\mbox{ interior sites of
}\mathcal {C}.
\end{equation*}

For $r>0$, let $\mathbb{D}_r$ denote the Euclidean disc of radius
$r$ centered at 0 and $\partial\mathbb{D}_r$ denote the boundary of
$\mathbb{D}_r$.  Write $\mathbb{D}:=\mathbb{D}_1$.  For $v\in
\mathbb{V}$, let $B(v,r)$ denote the set of hexagons of
$\mathbb{\mathbb{H}}$ that are contained in $v+\mathbb{D}_r$.  We
will sometimes see $B(v,r)$ as a union of these closed hexagons.
For $B(v,r)$, denote by $\partial B(v,r)$ its (topological) boundary
and by $\Delta B(v,r)$ its external site boundary (i.e., the set of
hexagons that do not belong to $B(v,r)$ but are adjacent to hexagons
in it). Write $B(r):=B(0,r)$.

Curves are equivalence classes of continuous functions from the unit
interval to $\mathbb{C}$, modulo monotonic reparametrizations.  Let
$\textrm{d}(\cdot,\cdot)$ denote the uniform metric on curves:
\begin{equation*}
\textrm{d}(\gamma_1,\gamma_2):=\inf\sup_{t\in
[0,1]}|\gamma_1(t)-\gamma_2(t)|,
\end{equation*}
where the infimum is taken over all choices of parametrizations of
$\gamma_1$ and $\gamma_2$ from the interval $[0,1]$.  The distance
between two closed sets of curves is defined by the induced
Hausdorff metric as follows:
\begin{equation}\label{e45}
\dist(\mathcal {F},\mathcal {F}'):=\inf\{\epsilon>0:\forall
\gamma\in \mathcal {F},\exists \gamma'\in\mathcal {F}'\mbox{ such
that }\textrm{d}(\gamma,\gamma')\leq\epsilon\mbox{ and vice
versa}\}.
\end{equation}

For critical site percolation on $\mathbb{T}$, we orient a cluster
boundary loop counterclockwise if it has blue sites on its inner
boundary and yellow sites on its outer boundary, otherwise we orient
it clockwise.  We say $B(R)$ has \textbf{monochromatic (blue)
boundary condition} if all the sites in $\Delta B(R)$ are blue.  In
\cite{2}, Camia and Newman showed the following well-known result
(see also Theorem 2 in \cite{32} for the case of a general Jordan
domain):

\begin{theorem}[\cite{2}]\label{t2}
As $\eta\rightarrow 0$, the collection of all cluster boundaries of
critical site percolation on $\eta\mathbb{T}$ in $\mathbb{D}$ with
monochromatic boundary conditions converges in distribution, under
the topology induced by metric (\ref{e45}), to a probability
distribution on collections of continuous nonsimple loops in
$\overline{\mathbb{D}}$.
\end{theorem}

Camia and Newman call the continuum nonsimple loop process in
Theorem \ref{t2} the full scaling limit of critical site
percolation, which is just the Conformal Loop Ensemble CLE$_6$ in
$\overline{\mathbb{D}}$.  The CLE$_\kappa$ for $8/3<\kappa<8$ is the
canonical conformally invariant measure on countably infinite
collections of noncrossing loops in a simply connected planar
domain, and is conjectured to correspond to the scaling limit of a
wide class of discrete lattice-based models, see \cite{16,17}.  In
the following, CLE$_6$ means CLE$_6$ in $\overline{\mathbb{D}}$.  We
denote by $\mathbb{P}$ the probability measure of CLE$_6$ and by
$\mathbb{E}$ the expectation with respect to $\mathbb{P}$.

Let $\mathcal {L}_k$ be the $k$th largest CLE$_6$ loop that
surrounds 0. Define $U_0:=\mathbb{D}$,  and let $U_k$ be the
connected component of the open set $\mathbb{D}\backslash\mathcal
{L}_k$ that contains 0.  If $D$ is a simply connected planar domain
with $0\in D$, the \textbf{conformal radius} of $D$ viewed from 0 is
defined to be $\CR(D):=|g'(0)|^{-1}$, where $g$ is any conformal map
from $D$ to $\mathbb{D}$ that sends 0 to 0.  For $k\in \mathbb{N}$,
define
$$B_k:=\log\CR(U_{k-1})-\log\CR(U_k).$$
Proposition 1 in \cite{15} says that $\{B_k\}_{k\in\mathbb{N}}$ are
i.i.d. random variables.  Furthermore, Schramm, Sheffield and Wilson
(see (3) and the remark following the statement of Theorem 1 in
\cite{15}) proved the following result for the moment generating
function of $B_k$, which is a key ingredient for the proof of our
main theorem.
\begin{theorem}[\cite{15}]\label{t3}
For $-\infty<\lambda<5/48$ and $k\in\mathbb{N}$,
\begin{equation*}
\mathbb{E}[\exp(\lambda
B_k)]=\frac{1}{2\cos(\pi\sqrt{1/9+4\lambda/3})}.
\end{equation*}
\end{theorem}
Theorem \ref{t3} implies the following corollary immediately.
\begin{corollary}\label{c2}
For $k\in \mathbb{N}$,
\begin{align}
&\mathbb{E}[B_k]=2\sqrt{3}\pi,\label{e8}\\
&\Var[B_k]=16\pi^2-12\sqrt{3}\pi.\label{e7}
\end{align}
\end{corollary}

To state next proposition we need more definitions.  For $1\leq
r<R$, let $A(r,R):=B(R)\backslash B(r)$.  Define
\begin{align*}
\rho(r,R)&:=\mbox{the maximal number of disjoint yellow circuits
surrounding 0 in }A(r,R),\\
N(r,R)&:=\mbox{the number of cluster boundary loops surrounding 0 in
$A(r,R)$,}\\
T'(r,R)&:=\inf \{T(\gamma):\gamma \mbox{ is a path connecting
$\partial B(r)$ and $\partial B(R)$}\}.
\end{align*}

$T'(r,R)$ satisfies two combinatorial properties as follows, the
first one is basically the same as (2.39) in \cite{10}, and the
second one can be derived from the first one and a ``color switching
trick".  We note that similar trick has appeared in \cite{24,23,25}.

\begin{proposition}\label{l1}
Suppose $1\leq r<R$.  Then $T'(r,R)$ satisfies the following
properties:
\begin{itemize}
\item $T'(r,R)=\rho(r,R)$.
\item Assume that
$B(R)$ has monochromatic (blue) boundary condition.  Then $T'(r,R)$
has the same distribution as $N(r,R)$.
\end{itemize}
\end{proposition}

\begin{figure}
\begin{center}
\includegraphics[height=0.4\textwidth]{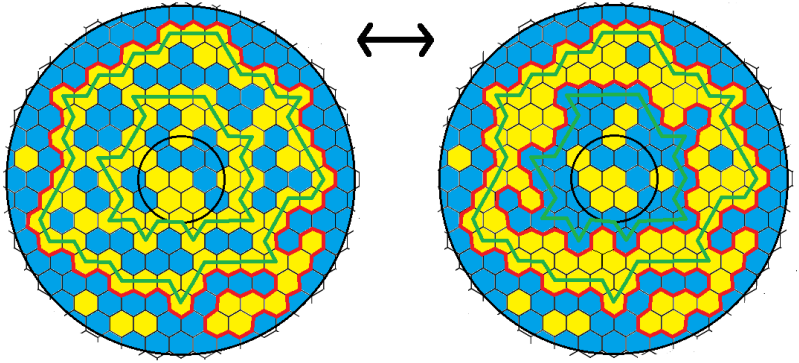}
\caption{A sketch of the color switching trick used in the proof of
Proposition \ref{l1}.}\label{fig1}
\end{center}
\end{figure}

\begin{proof}
The proof of the equation is essentially the same as that of (2.39)
in \cite{10} for critical FPP on $\mathbb{Z}^2$.  For completeness,
we give the proof for our setting. If $\rho(r,R)=0$, then there
exists a blue path connecting $\partial B(R)$ and $\partial B(r)$,
and we get $T'(r,R)=0$.  We assume $\rho(r,R)>0$ in the following.

First we show $T'(r,R)\geq \rho(r,R)$.  This inequality is trivial,
since if there are $\rho$ disjoint yellow circuits surrounding 0 in
$A(r,R)$, any path connecting the two boundary pieces of $A(r,R)$
must intersect these circuits.

Next we show the converse inequality, $T'(r,R)\leq \rho(r,R)$. We
shall construct a path $\gamma$ connecting $\partial B(R)$ and
$\partial B(r)$ such that $T(\gamma)=\rho(r,R)$, which implies
$T'(r,R)\leq \rho(r,R)$ immediately.  Let $D_0:=B(R)$.  Take
$\mathcal {C}_1$ the outermost yellow circuit surrounding 0 in $D_0$
and let $D_1$ be the component of $D_0\backslash \mathcal {C}_1$
that contains 0, then take $\mathcal {C}_2$ the outermost yellow
circuit surrounding 0 in $D_1$ and let $D_2$ be the component of
$D_1\backslash \mathcal {C}_2$ that contains 0, and so on.  The
process stops after $\rho=\rho(r,R)$ steps.  It is easy to see that
$\mathcal {C}_k\subset A(r,R)$ for $1\leq k\leq \rho$ and there
exists no yellow circuit surrounding 0 in $D_{\rho}\backslash B(r)$.
Thus we can take a blue path $\gamma_{\rho}$ connecting $\Delta
B(r)$ and a site $v_{\rho}$ in $\mathcal {C}_{\rho}$ (note that
$\gamma_{\rho}$ may be empty if $\mathcal {C}_{\rho}$ intersects
$\Delta B(r)$, similar case may occur below), then a blue path
$\gamma_{\rho-1}$ connecting $v_{\rho}$ and a site $v_{\rho-1}$ in
$\mathcal {C}_{\rho-1}$ since $\mathcal {C}_{\rho}$ is the outermost
yellow circuit in $D_{\rho-1}$, and so on.  The process stops after
$\rho+1$ steps, and $\gamma_{0}$ is a blue path connecting $v_1\in
\mathcal {C}_1$ and the inner site boundary of $B(R)$.  Let
$\gamma=\gamma_0v_1\ldots\gamma_{\rho-1}v_{\rho}\gamma_{\rho}$.
Clearly $\gamma$ connects $\partial B(R)$ and $\partial B(r)$ and
$T(\gamma)=\rho(r,R)$.

We now turn to the proof of the second property.  See Fig.
\ref{fig1} for an illustration of the following argument.  Since
$B(R)$ has monochromatic (blue) boundary condition, it is easy to
see that $N(r,R)$ equals the maximal number of disjoint circuits
$\mathcal {C}_1',\mathcal {C}_2',\ldots$ surrounding 0 in $A(r,R)$
with alternating colors (yellow, blue, yellow, blue, $\ldots$) and
$\overline{\mathcal {C}_1'}\supset\overline{\mathcal
{C}_2'}\supset\ldots$.  For any fixed $n\in \mathbb{N}\cup\{0\}$, we
shall construct a bijection between the sets
$\{\omega:\rho(r,R)=n\}$ and $\{\omega':N(r,R)=n\}$.  Given a
configuration $\omega\in\{\rho(r,R)=n\}$, we construct a sequence of
yellow circuits $\mathcal {C}_1,\ldots,\mathcal {C}_n$ from outside
to inside and a sequence of connected domains $D_0,\ldots,D_n$ as in
the proof of the first property.  If $n$ is odd, we switch the
colors of the sites in $D_1\backslash D_2,\ldots,D_{n-2}\backslash
D_{n-1},D_n$; if $n$ is even, we switch the colors of the sites in
$D_1\backslash D_2,\ldots,D_{n-1}\backslash D_n.$  Denote this
transformation by $f_n$.  Then $f_n(\omega)\in\{N(r,R)=n\}$.  This
can be seen as follows: When $n$ is odd, after the transformation
$f_n$, the color of $\mathcal {C}_1,\mathcal {C}_3,\ldots,\mathcal
{C}_n$ is invariant and the color of $\mathcal {C}_2,\mathcal
{C}_4,\ldots,\mathcal {C}_{n-1}$ is switched to blue.  Then we can
find a cluster boundary loop surrounding 0 in each $D_j\backslash
D_{j+1},0\leq j\leq n-1$.  Furthermore, none of these domains has
two such loops, otherwise it would produce a yellow circuit
surrounding 0 between two successive circuits $\mathcal {C}_j$ and
$\mathcal {C}_{j+1}$ or outside $\mathcal {C}_1$ in $D_0$ in the
original configuration, which contradicts our construction.  There
does not exist cluster boundary loop surrounding 0 in $D_n\backslash
B(r)$ since $\omega\in\{\rho(r,R)=n\}$.  Hence,
$f_n(\omega)\in\{N(r,R)=n\}$.  The argument is similar when $n$ is
even.

Given a configuration $\omega'\in\{N(r,R)=n\}$, similarly as above,
we let $D_0':=B(R)$, take $\mathcal {C}_1'$ the outermost yellow
circuit surrounding 0 in $D_0'$ and let $D_1'$ be the component of
$D_0'\backslash \mathcal {C}_1'$ that contains 0, then take
$\mathcal {C}_2'$ the outermost blue circuit surrounding 0 in $D_1'$
and let $D_2'$ be the component of $D_1'\backslash \mathcal {C}_2'$
that contains 0, and so on.  The process stops after $n$ steps.  If
$n$ is odd, we switch the colors of the sites in $D_1'\backslash
D_2',\ldots,D_{n-2}'\backslash D_{n-1}',D_n'$; if $n$ is even, we
switch the colors of the sites in $D_1'\backslash
D_2',\ldots,D_{n-1}'\backslash D_n'.$  Note that $\omega'$ is
transformed to $\omega\in\{\rho(r,R)=n\}$ and this transformation is
just $f_n^{-1}$.   This can be seen as follows: When $n$ is odd,
after the transformation, the color of $\mathcal {C}_1',\mathcal
{C}_3',\ldots,\mathcal {C}_n'$ is invariant and the color of
$\mathcal {C}_2',\mathcal {C}_4',\ldots,\mathcal {C}_{n-1}'$ is
switched to yellow, and $\mathcal {C}_j'$ is the outermost yellow
circuit surrounding 0 in $D_{j-1}',1\leq j\leq n$.  Furthermore,
$\rho(r,R)=n$, otherwise it would produce more than $n$ cluster
boundary loops surrounding 0 in $A(r,R)$ before the transformation.
Therefore, this transformation is just $f_n^{-1}$ by the
construction and the definition of $f_n$.  The argument is similar
when $n$ is even.  Then the bijection $f_n$ between
$\{\omega:\rho(r,R)=n\}$ and $\{\omega':N(r,R)=n\}$ is constructed
for each $n\in \mathbb{N}\cup\{0\}$, which completes the proof since
we use the uniform measure.
\end{proof}

The following lemma gives upper large deviation bound for $T'(r,R)$:
\begin{lemma}[Corollary 2.3 in \cite{20}]\label{l10}
There exist constants $C_1,C_2>0$ and $K>1$, such that for all
$1\leq r< R$ and $x\geq K\log_2(R/r)$,
\begin{equation*}
P[T'(r,R)\geq x]\leq C_1\exp(-C_2x).
\end{equation*}
\end{lemma}

In this paper, $C,C_1,C_2,\ldots$ denote positive finite constants
that may change from line to line according to the context.

\section{Proofs of the main results}

\subsection{Scaling limits of annulus times}\label{s31}

For $0<\epsilon<1$, denote by $N(\epsilon)$ the number of CLE$_6$
loops surrounding 0 in $\overline{\mathbb{D}}\backslash
\overline{\mathbb{D}}_{\epsilon}$.  Camia and Newman's full scaling
limit allows us to derive a scaling limit of $T'(\rho r,\rho R)$ as
$\rho\rightarrow\infty$:

\begin{proposition}\label{l3}
Suppose $1\leq r<R,\rho>0$ and $k\in\mathbb{N}$.  Assume that
$B(\rho R)$ has monochromatic (blue) boundary condition.  As
$\rho\rightarrow\infty$, we have
\begin{align}
&T'(\rho r,\rho R)\stackrel{d}\longrightarrow N(r/R),\label{e16}\\
&E[(T'(\rho r,\rho R))^k]\rightarrow
\mathbb{E}[(N(r/R))^k].\label{e17}
\end{align}
\end{proposition}
\begin{proof}
Let
\begin{align*}
\mathcal {F}_{\rho}&:=\mbox{the collection of cluster boundary loops
surrounding 0 in $A(\rho r,\rho R)$}\\
&~~~~~\mbox{scaled by $1/(\rho R)$},\\
\mathcal {F}&:=\mbox{the collection of CLE$_6$ loops surrounding 0
in $\overline{\mathbb{D}}\backslash \overline{\mathbb{D}}_{r/R}$}.
\end{align*}

Define event
\begin{align*}
\mathcal {A}_{\epsilon,\rho}:=\{\exists \mathcal {L}\in \mathcal
{F}_{\rho}\mbox{ such that }\mathcal {L}\cap
(\mathbb{D}_{(1+\epsilon)r}\backslash
\mathbb{D}_{(1-\epsilon)r})\neq\emptyset\}.
\end{align*}
Assume that $\mathcal {A}_{\epsilon,\rho}$ holds and $\rho$ is large
enough (depending on $\epsilon$).  Then we have a polychromatic
3-arm event from a ball of radius $3\epsilon$ centered at a point
$z\in\partial \mathbb{D}_{(1-\epsilon)r}$ to a distance of unit
order in $\overline{\mathbb{D}}\backslash
\mathbb{D}_{(1-\epsilon)r}$.  For a fixed $z\in\partial
\mathbb{D}_{(1-\epsilon)r}$, the corresponding 3-arm event happens
with probability at most $O(\epsilon^2)$ (see e.g. Lemma 6.8 in
\cite{18}).  From this one easily obtains $P[\mathcal
{A}_{\epsilon,\rho}]\leq O(\epsilon)$.  Then Theorem \ref{t2}
implies that $\mathcal {F}_{\rho}$ converges in distribution to
$\mathcal {F}$ as $\rho\rightarrow\infty$. Because of the choice of
topology, we can find coupled versions of $\mathcal {F}_{\rho}$ and
$\mathcal {F}$ on the same probability space such that
$\dist(\mathcal {F}_{\rho},\mathcal {F})\rightarrow 0$ a.s. as
$\rho\rightarrow \infty$ (see e.g. Theorem 6.7 of \cite{Bill},
similar couplings were used in \cite{2}).

Now let us show that for large $\rho$, with high probability the
distance $\textrm{d}(\cdot,\cdot)$ between the loops in $\mathcal
{F}_{\rho}$ is not very small.  Define event
\begin{equation*}
\mathcal {B}_{\epsilon,\rho}:=\{\exists \mathcal {L}_1,\mathcal
{L}_2\in \mathcal {F}_{\rho}\mbox{ such that }\textrm{d}(\mathcal
{L}_1,\mathcal {L}_2)<\epsilon\}.
\end{equation*}
Assume that $\mathcal {B}_{\epsilon,\rho}$ holds and $\rho$ is large
enough (depending on $\epsilon$).  Suppose $\mathcal {L}_2$ is in
the interior of $\mathcal {L}_1$.  Note that the polychromatic
half-plane 3-arm exponent $\beta_3$ is 2 and the polychromatic plane
6-arm exponent $\alpha_6$ is larger that 2 (see e.g. \cite{14}). Let
$0<\alpha<1-2/\alpha_6$ be a fixed number.  If $\mathcal {L}_1$
comes $\epsilon^{\alpha}/2$-close to $\partial \mathbb{D}$, then we
have a half-plane 3-arm event from a ball of radius
$2\epsilon^{\alpha}$ centered on a point in $\partial \mathbb{D}$ to
a distance of unit order in $\overline{\mathbb{D}}$; otherwise, we
have a polychromatic plane 6-arm event from radius $\epsilon$ to
$\epsilon^{\alpha}/2$ in $\overline{\mathbb{D}}$.  So, the event
$\mathcal {B}_{\epsilon,\rho}$ happens with probability at most
$O(\epsilon^{-\alpha}\epsilon^{\alpha\beta_3}+\epsilon^{-2}\epsilon^{(1-\alpha)(\alpha_6+o(1))})=O(\epsilon^{\alpha}+\epsilon^{(1-\alpha)(\alpha_6+o(1))-2})$.
This implies that in the above coupling, the number of loops in
$\mathcal {F}_{\rho}$ converges in probability to the number of
loops in $\mathcal {F}$ as $\rho\rightarrow\infty$.  Then we get
$N(\rho r,\rho R)\stackrel{d}\longrightarrow N(r/R)$ immediately.
From Proposition \ref{l1}, we obtain (\ref{e16}).  Lemma \ref{l10}
and (\ref{e16}) imply (\ref{e17}).
\end{proof}

The key ingredient Theorem \ref{t3} enables us to obtain the
following limit theorem for $N(\epsilon)$:
\begin{proposition}\label{l2}
\begin{align}
&\lim_{\epsilon\rightarrow 0}\frac{\mathbb{E}[N(\epsilon)]}{\log
(1/\epsilon)}=\frac{1}{2\sqrt{3}\pi},~~\lim_{\epsilon\rightarrow
0}\frac{N(\epsilon)}{\log(
1/\epsilon)}=\frac{1}{2\sqrt{3}\pi}~~a.s.,\label{e9}\\
&\lim_{\epsilon\rightarrow 0}\frac{\Var[N(\epsilon)]}{\log(
1/\epsilon)}=\frac{2}{3\sqrt{3}\pi}-\frac{1}{2\pi^2}.\label{e10}
\end{align}
\end{proposition}
\begin{proof}
From Proposition 1 in \cite{15}, we know that
$\{B_k\}_{k\in\mathbb{N}}$ are i.i.d. random variables.  It is clear
that $S_n:=\sum_{k=1}^{n}B_k=-\log\CR(U_n),n\geq 1$ is a renewal
process.  We let $N_t:=\inf\{n:S_n>t\}$ denote the number of
renewals in $[0,t]$.  Schwarz Lemma and the Koebe $1/4$ Theorem (see
e.g. Lemma 2.1 and Theorem 3.17 in \cite{12}) imply that
\begin{equation*}
\frac{\CR(U_n)}{4}\leq\dist(0,\mathcal {L}_n)\leq \CR(U_n).
\end{equation*}
Hence, $\CR(U_{N(\epsilon)})\geq\dist(0,\mathcal
{L}_{N(\epsilon)})>\epsilon$ and $\CR(U_{N(\epsilon)+1})\leq
4\dist(0,\mathcal {L}_{N(\epsilon)+1})\leq 4\epsilon$.  Then we have
\begin{equation}\label{e11}
N_{\log(1/(5\epsilon))}-1\leq N(\epsilon)<N_{\log(1/\epsilon)}.
\end{equation}
Applying the elementary renewal theorem and law of large numbers for
renewal processes (see e.g. (4.1) and (4.2) in Section 3.4 in
\cite{27}), we know
\begin{equation}
\lim_{t\rightarrow\infty}\frac{\mathbb{E}[N_t]}{t}\rightarrow\frac{1}{\mathbb{E}[B_1]},~~
\lim_{t\rightarrow\infty}\frac{N_t}{t}\rightarrow\frac{1}{\mathbb{E}[B_1]}
~~a.s.\label{e12}
\end{equation}
Combining (\ref{e8}), (\ref{e11}) and (\ref{e12}), we obtain
(\ref{e9}).

There is a variance analogue of the elementary renewal theorem, see
e.g. (1.6) in \cite{29}.  From this we know
\begin{equation}\label{e13}
\lim_{t\rightarrow\infty}\frac{\Var [N_t]}{t}=\frac{\Var
[B_1]}{(\mathbb{E}[B_1])^3}.
\end{equation}
The triangle inequality for the norm
$\|\cdot\|_2=\sqrt{\mathbb{E}[|\cdot|^2]}$ and (\ref{e11}) give
\begin{align*}
\left|\sqrt{\Var[N(\epsilon)]}-\sqrt{\Var[N_{\log(1/\epsilon)}]}\right|&\leq\sqrt{\Var[N(\epsilon)-N_{\log(1/\epsilon)}]}\\
&\leq\sqrt{\mathbb{E}[(N_{\log(1/\epsilon)}-N_{\log(1/(5\epsilon))}+1)^2]}.
\end{align*}
It is easy to see that the number of renewals in each time interval
of length $\log 5$ can be uniformly dominated by a positive random
variable with an exponential tail.  Therefore, there exists a
constant $C>0$ such that for all $0<\epsilon<1$,
\begin{equation*}
\left|\sqrt{\Var[N(\epsilon)]}-\sqrt{\Var[N_{\log(1/\epsilon)}]}\right|\leq
C.
\end{equation*}
Combining this inequality, (\ref{e13}) and Corollary \ref{c2}, we
get (\ref{e10}).
\end{proof}
\begin{remark}
In fact, if one uses second-order approximations to the
expectation-time and variance-time curves for renewal processes (see
e.g. (1.4) and the equation just above Section 1.5 in \cite{29}) in
the above proof, one obtains that there exists a constant $C>0$ such
that for all $0<\epsilon<1$,
\begin{equation*}
\left|\mathbb{E}[N(\epsilon)]-\frac{\log(1/\epsilon)}{2\sqrt{3}\pi}\right|\leq
C,~~\left|\Var[N(\epsilon)]-\frac{\log(
1/\epsilon)}{\frac{2}{3\sqrt{3}\pi}-\frac{1}{2\pi^2}}\right|\leq C.
\end{equation*}
We note that the first inequality above has been proved in \cite{28}
(see (3.8) in \cite{28}).
\end{remark}

\subsection{Limit results for $c_n$ and $E[c_n]$}\label{s32}
Recall that $c_n$ is the passage time from 0 to $\partial B(n)$. In
order to prove our limit results for $c_n$ and $E[c_n]$, we need
some lemmas.  In \cite{20}, we proved the following result.
\begin{lemma}[Lemma 2.5 in \cite{20}]\label{l6}
\begin{equation*}
\lim_{n\rightarrow\infty}\frac{c_n}{E[c_n]}=1~~a.s.
\end{equation*}
\end{lemma}

Lemma \ref{l10} together with RSW and FKG (see e.g. \cite{5}), gives
the following lemma.  Note that (\ref{e20}) was first proved in
\cite{30}.
\begin{lemma}\label{l9}
There exist $C_1,C_2>0$ such that for all $1\leq r\leq R/2$,
\begin{equation}\label{e22}
C_1\log(R/r)\leq E[T'(r,R)]\leq C_2\log(R/r).
\end{equation}
In particular, for all $n\geq 2$,
\begin{equation}\label{e20}
C_1\log n\leq E[c_n]\leq C_2\log n.
\end{equation}
\end{lemma}

For $x\geq 0$, denote by $S_x$ the maximum number of disjoint yellow
circuits that surround 0 and intersect with $\partial B(2^x)$. Using
RSW, FKG and BK inequality, it is easy to obtain the following
result.

\begin{lemma}\label{l14}
There exists a constant $C>0$, such that for all $x,t\geq 0$,
\begin{equation}\label{e38}
P[S_x\geq t]\leq \exp(-Ct).
\end{equation}
Hence, there is a constant $C_0>0$ (independent of $x$), such that
\begin{equation}\label{e31}
E[S_x]\leq C_0.
\end{equation}
\end{lemma}

Now we are ready to prove the limit result for $c_n$ and $E[c_n]$:
\begin{proposition}\label{l4}
\begin{align}
&\lim_{n\rightarrow\infty}\frac{c_n}{\log
n}=\frac{1}{2\sqrt{3}\pi}~~a.s.,\label{e15}\\
&\lim_{n\rightarrow\infty}\frac{E[c_n]}{\log
n}=\frac{1}{2\sqrt{3}\pi}.\label{e14}
\end{align}
\end{proposition}

\begin{proof}
Note that Lemma \ref{l6} and (\ref{e14}) imply (\ref{e15}),  so it
suffices to show (\ref{e14}).  We write
$T_{k,j}:=T'(2^{k(j-1)},2^{kj})$.   By Proposition \ref{l1}, for
$k\in\mathbb{N}$ and $n\geq 2^k$,  it is clear that
\begin{equation}\label{e19}
-1+\sum_{j=1}^{\lfloor\log_{2^k} n\rfloor} T_{k,j}\leq c_n\leq
1+\rho(1,n)\leq 1+\sum_{j=1}^{\lfloor\log_{2^k} n\rfloor+1}
(T_{k,j}+S_{kj}).
\end{equation}
Combining this with (\ref{e22}) and (\ref{e31}), we obtain that
there exists $C_1,C_2>0$ such that
\begin{equation*}
-1+\sum_{j=1}^{\lfloor\log_{2^k} n\rfloor} E[T_{k,j}]\leq E[c_n]\leq
C_1\log_{2^k} n+C_2k+\sum_{j=1}^{\lfloor\log_{2^k} n\rfloor}
E[T_{k,j}].
\end{equation*}
This and (\ref{e20}) imply that for each $0<\epsilon<1$, there
exists $k_0(\epsilon)>0$, such that for each $k\geq k_0$, for $n$
sufficiently large (depending on $k$),
\begin{equation}\label{e18}
1-\epsilon\leq\frac{\sum_{j=1}^{\lfloor\log_{2^k} n\rfloor}
E[T_{k,j}]}{E[c_n]}\leq 1+\epsilon.
\end{equation}
By the convergence of the Ces\`{a}ro mean and (\ref{e17}), we obtain
\begin{equation}\label{e1}
\lim_{m\rightarrow\infty}\frac{\sum_{j=1}^m
E[T_{k,j}]}{m}=E[N(1/2^k)].
\end{equation}
Then (\ref{e14}) follows from (\ref{e9}), (\ref{e18}) and
(\ref{e1}):
\begin{equation*}
\lim_{n\rightarrow\infty}\frac{E[c_n]}{\log
n}=\lim_{k\rightarrow\infty}\lim_{n\rightarrow\infty}\frac{\sum_{j=1}^{\lfloor\log_{2^k}
n\rfloor} E[T_{k,j}]}{\log
n}=\lim_{k\rightarrow\infty}\frac{E[N(1/2^k)]}{\log
(2^k)}=\frac{1}{2\sqrt{3}\pi}.
\end{equation*}
\end{proof}

\subsection{Limit result for $\Var [c_n]$}\label{s33}
The proof of the result relating to variance turns out to be more
involved than that to expectation.  We shall use the martingale
approach in \cite{10}, with a modification by introducing an
intermediate scale. Let us mention that an analogous martingale
method was used in \cite{33} to prove a CLT for winding angles of
arms in 2D critical percolation.  We start with some definitions.

For $j\in \mathbb{N}\cup\{0\}$, define annulus
\begin{equation*}
A(j):=A(2^j,2^{j+1}).
\end{equation*}
Furthermore, define
\begin{align*}
&m(j):=\inf\{k\geq j: A(k)\mbox{ contains a blue circuit
surrounding 0}\},\\
&\mathcal {C}_j:=\mbox{the innermost blue circuit surrounding 0
in }A(m(j)),\\
&\mathscr{F}_j:=\sigma\mbox{-field generated by }\{t(v):v\in
\overline{\mathcal {C}}_j\}.
\end{align*}
For all $j\geq 1$, denote by $\mathcal {C}_{-j}$ the origin and by
$\mathscr{F}_{-j}$ the trivial $\sigma$-field.  For
$k,q\in\mathbb{N}$, write
\begin{equation*}
T(0,\mathcal {C}_{kq})-E[T(0,\mathcal
{C}_{kq})]=\sum_{j=0}^q\left(E[T(0,\mathcal
{C}_{kq})|\mathscr{F}_{kj}]-E[T(0,\mathcal
{C}_{kq})|\mathscr{F}_{k(j-1)}]\right):=\sum_{j=0}^q\Delta_{k,j}.
\end{equation*}
Then $\{\Delta_{k,j}\}_{0\leq j\leq q}$ is an
$\mathscr{F}_{kj}$-martingale increment sequence.  Hence,
\begin{equation}\label{e42}
\Var[T(0,\mathcal {C}_{kq})]=\sum_{j=0}^qE[\Delta_{k,j}^2].
\end{equation}
We will use this sum to estimate $\Var[c_n]$.

Let $(\Omega',\mathscr{F}',P')$ be a copy of
$(\Omega,\mathscr{F},P)$.  Denote by $E'$ the expectation with
respect to $P'$, and by $\omega'$ a sample point in $\Omega'$.  Let
$T(\cdot,\cdot)(\omega),m(j,\omega)$ and $\mathcal {C}_j(\omega)$
denote the quantities defined before, but with explicit dependence
on $\omega$. Define $l(j,\omega,\omega'):=m(m(j,\omega)+1,\omega')$.
We need the following results, which were proved in \cite{10}.  Note
that (\ref{e27}) follows from RSW and FKG, and the proof of
(\ref{e28}) is essentially the same as that of Lemma 2 in \cite{10}.
\begin{lemma}[\cite{10}]\label{l7}
(i) There exists $C>0$, such that for all $j,t\in\mathbb{N}$, we
have
\begin{equation}\label{e27}
P[m(j)\geq j+t]\leq \exp(-Ct).
\end{equation}
(ii) For $j\geq 0$ and $k\geq 1$, $\Delta_{k,j}$ does not depend on
$q$.  Furthermore,
\begin{align}
\Delta_{k,j}(\omega)=&T(\mathcal {C}_{k(j-1)}(\omega),\mathcal
{C}_{kj}(\omega))(\omega)+E'[T(\mathcal {C}_{kj}(\omega),\mathcal
{C}_{l(kj,\omega,\omega')}(\omega'))(\omega')]\nonumber\\
&-E'[T(\mathcal {C}_{k(j-1)}(\omega),\mathcal
{C}_{l(kj,\omega,\omega')}(\omega'))(\omega')].\label{e28}
\end{align}
\end{lemma}

Next we bound the variances of $\Delta_{k,j}$ and $T(0,\mathcal
{C}_j)$.
\begin{lemma}\label{l11}
There exist $C_1,C_2>0$ such that for all $j,k\geq 1$,
\begin{align}
&C_1k\leq E[\Delta_{k,j}^2]\leq C_2k,\label{e40}\\
&C_1j\leq \Var[T(0,\mathcal {C}_j)]\leq C_2j.\label{e41}
\end{align}
\end{lemma}

\begin{proof}
From (\ref{e42}), we know that (\ref{e41}) is just (2.32) in
\cite{10}.  The proof of (2.32) in \cite{10} implies that there
exist universal $C_1,C_2>0$ such that for all $n\in\mathbb{N}$,
\begin{equation*}
C_1\leq E[\Delta_{1,n}^2]\leq C_2.
\end{equation*}
Since
\begin{equation*}
E[\Delta_{k,j}^2]=\sum_{n=kj-k+1}^{kj} E[\Delta_{1,n}^2],
\end{equation*}
then (\ref{e40}) follows.
\end{proof}

The next lemma implies that for fixed $k$ and large $q$, if
$2^{kq}\leq n\leq 2^{k(q+1)}$,  the variance of $c_n$ is well
approximated by that of $T(0,\mathcal {C}_{kq})$.
\begin{lemma}\label{l13}
There exists $C>0$, such that for all $q\geq k\geq 1$ and
$2^{kq}\leq n\leq 2^{k(q+1)}$,
\begin{equation}\label{e37}
\left|\Var[T(0,\mathcal {C}_{kq})]-\Var[c_n]\right|\leq
C\sqrt{q}k^{3/2}.
\end{equation}
Furthermore, there exist $C_1,C_2>0$ such that for all $n\geq 2$,
\begin{equation}\label{e43}
C_1\log n\leq\Var[c_n]\leq C_2\log n.
\end{equation}
\end{lemma}

\begin{proof}
Recall that $S_j$ denotes the maximum number of disjoint yellow
circuits that surround 0 and intersect with $\partial B(2^j)$.
Applying Lemma \ref{l10}, (\ref{e38}) and (\ref{e27}), there exist
$C_1,C_2>0$, such that for all $x\geq 3Kk$, where $K$ is from Lemma
\ref{l10}, we have
\begin{align*}
&P[|T(0,\mathcal {C}_{kq})-c_n|\geq x]\\
&\leq P[m(kq)\geq kq+\lfloor
x/(2K)\rfloor]+P[T'(2^{kq},2^{kq+\lfloor
x/(2K)\rfloor})\geq x/2 ]+P[S_{kq}\geq x/2]\\
&\leq C_1\exp(-C_2x).
\end{align*}
Then there is a universal $C_3>0$, such that
\begin{equation*}
E[|T(0,\mathcal {C}_{kq})-c_n|^2]\leq C_3k^2.
\end{equation*}
Therefore, by the triangle inequality for the norm
$\|\cdot\|_2=\sqrt{E[|\cdot|^2]}$, we have
\begin{align*}
\left|\sqrt{\Var[T(0,\mathcal
{C}_{kq})]}-\sqrt{\Var[c_n]}\right|&\leq \sqrt{\Var[T(0,\mathcal
{C}_{kq})-c_n]}\leq \sqrt{C_3}k.
\end{align*}
Combining this and (\ref{e41}) gives (\ref{e37}).

(\ref{e43}) follows from (\ref{e41}) and (\ref{e37}) immediately.
\end{proof}

Recall that $T_{k,j}:=T'(2^{k(j-1)},2^{kj})$.  The following is a
key estimate for $T_{k,j}$,  which says that for large $k$, the
variance of  $T_{k,j}$ is well approximated by that of
$\Delta_{k,j}$ (since $E[\Delta_{k,j}^2]\asymp k$, from
(\ref{e40})).

\begin{lemma}\label{l8}
There exists a constant $C>1$, such that for all $k,j\in\mathbb{N}$,
\begin{align*}
\left|\Var[T_{k,j}]-E[\Delta_{k,j}^2]\right|\leq C\sqrt{k}.
\end{align*}
\end{lemma}
\begin{proof}
We claim that there exist universal constants $C_1,C_2>0$ such that
for all $k,j\in\mathbb{N}$ and $x>0$,
\begin{equation}\label{e24}
P[|T_{k,j}-E[T_{k,j}]-\Delta_{k,j}|\geq x]\leq C_1\exp(-C_2x).
\end{equation}
Therefore, there is a $C_3>0$ such that
\begin{equation*}
E[|T_{k,j}-E[T_{k,j}]-\Delta_{k,j}|^2]\leq C_3.
\end{equation*}
Then we obtain
\begin{equation*}
|\{\Var[T_{k,j}]\}^{1/2}-\{\Var[\Delta_{k,j}]\}^{1/2}|\leq
\{E[|T_{k,j}-E[T_{k,j}]-\Delta_{k,j}|^2]\}^{1/2}\leq \sqrt{C_3}.
\end{equation*}
This and (\ref{e40}) imply the desired result:
\begin{equation*}
\left|\Var[T_{k,j}]-E[\Delta_{k,j}^2]\right|\leq
2\sqrt{C_3E[\Delta_{k,j}^2]}+C_3\leq C\sqrt{k}.
\end{equation*}

We now show our claim (\ref{e24}).  By (\ref{e28}), proving
(\ref{e24}) boils down to proving that there exist
$C_4,C_5,C_6,C_7>0$ such that
\begin{align}
&P[|T_{k,j}-T(\mathcal {C}_{k(j-1)},\mathcal {C}_{kj})|\geq
x]\leq C_4\exp(-C_5x),\label{e25}\\
&P[|E[T_{k,j}]+E'[T(\mathcal {C}_{kj}(\omega),\mathcal
{C}_{l(kj,\omega,\omega')}(\omega'))]-E'[T(\mathcal
{C}_{k(j-1)}(\omega),\mathcal
{C}_{l(kj,\omega,\omega')}(\omega'))]|\geq x]\nonumber\\
&~~\leq C_6\exp(-C_7x).\label{e26}
\end{align}

We first prove (\ref{e25}).  Let $C_8=1/(4K)$, where $K$ is from
Lemma \ref{l10}.  There exist $C_9,C_{10},C_{11},C_{12}>0$ such that
if $0<x\leq k/C_8$,
\begin{align*}
&P[|T_{k,j}-T(\mathcal {C}_{k(j-1)},\mathcal {C}_{kj})|\geq
x]\\
&\leq P[m(k(j-1))\geq k(j-1)+\lfloor C_8x\rfloor]+P[m(kj-\lfloor
C_8x\rfloor)\geq
kj]+P[m(kj)\geq kj+\lfloor C_8x\rfloor]\\
&~~~~+P[T'(2^{k(j-1)},2^{k(j-1)+\lfloor C_8x\rfloor})\geq
x/2]+P[T'(2^{kj-\lfloor C_8x\rfloor},2^{kj+\lfloor C_8x\rfloor})\geq x/2]\\
&\leq 3C_9\exp(-C_{10}x)+2C_{11}\exp(-C_{12}x)~~\mbox{ by
(\ref{e27}) and Lemma \ref{l10}};
\end{align*}
if $x\geq k/C_8$,
\begin{align*}
&P[|T_{k,j}-T(\mathcal {C}_{k(j-1)},\mathcal {C}_{kj})|\geq
x]\\
&~~~~\leq P[m(kj)\geq kj+\lfloor
C_8x\rfloor]+P[T'(2^{k(j-1)},2^{kj+\lfloor C_8x\rfloor})\geq
x]\\
&~~~~\leq C_9\exp(-C_{10}x)+C_{11}\exp(-C_{12}x)~~\mbox{ by
(\ref{e27}) and Lemma \ref{l10}}.
\end{align*}
Then (\ref{e25}) follows immediately.

Now let us show (\ref{e26}), which follows from the two inequalities
below:
\begin{align}
&E'[T(\mathcal {C}_{kj}(\omega),\mathcal
{C}_{l(kj,\omega,\omega')}(\omega'))]\leq C_{13}~~a.s.,\label{e29}\\
&P[|E'[T_{k,j}(\omega')]-E'[T(\mathcal {C}_{k(j-1)}(\omega),\mathcal
{C}_{l(kj,\omega,\omega')}(\omega'))]|\geq x]\leq
C_{14}\exp(-C_{15}x),\label{e30}
\end{align}
where $C_{13},C_{14},C_{15}>0$ are universal constants.  Using
(\ref{e27}) and Lemma \ref{l10} again, we know there exist
$C_{16},C_{17}>0$ such that for all $\mathcal {C}_{kj}(\omega)$,
\begin{align*}
&P'[T(\mathcal {C}_{kj}(\omega),\mathcal
{C}_{l(kj,\omega,\omega')}(\omega'))\geq x]\\
&~~~~\leq P'[l(kj,\omega,\omega')\geq m(kj,\omega)+\lfloor
C_8x\rfloor]+P'[T'(2^{m(kj,\omega)},2^{m(kj,\omega)+\lfloor C_8x\rfloor})\geq x]\\
&~~~~\leq C_{16}\exp(-C_{17}x),
\end{align*}
which implies (\ref{e29}).

It remains to show (\ref{e30}).  Similarly to the proof of
(\ref{e25}), one can show that there exist $C_{18},C_{19}>0$, such
that
\begin{equation*}
P'[|T_{m(k(j-1),\omega),m(kj,\omega)+1}(\omega')-T(\mathcal
{C}_{k(j-1)}(\omega),\mathcal
{C}_{l(kj,\omega,\omega')}(\omega'))|\geq x]\leq
C_{18}\exp(-C_{19}x),
\end{equation*}
which implies that there is a universal $C_{20}>0$, such that
\begin{equation*}
\left|E'[T_{m(k(j-1),\omega),m(kj,\omega)+1}(\omega')]-E'[T(\mathcal
{C}_{k(j-1)}(\omega),\mathcal
{C}_{l(kj,\omega,\omega')}(\omega'))]\right|\leq C_{20}.
\end{equation*}
From this proving (\ref{e30}) boils down to proving that there are
$C_{21},C_{22}>0$ such that
\begin{equation}\label{e33}
P[E'[T_{k,j}(\omega')]-E'[T_{m(k(j-1),\omega),m(kj,\omega)+1}(\omega')]\geq
x]\leq C_{21}\exp(-C_{22}x).
\end{equation}
Combining (\ref{e22}) and (\ref{e31}) gives that there exists
$C_{24}>0$, such that for all $x\geq 1$ and $j\geq 0$,
\begin{equation}\label{e32}
E[T'(2^j,2^{j+x+1})+S_{j}+S_{j+x+1}]\leq x/(2C_{24}).
\end{equation}
Then, using (\ref{e32}) and (\ref{e27}), we get that there exist
$C_{25},C_{26}>0$ such that if $1/C_{24}\leq x\leq k/C_{24}$,
\begin{align*}
&P[|E'[T_{k,j}(\omega')]-E'[T_{m(k(j-1),\omega),m(kj,\omega)+1}(\omega')]|\geq
x]\\
&~~~~\leq P[m(k(j-1))\geq k(j-1)+\lfloor C_{24}x\rfloor]+P[m(kj)\geq
kj+\lfloor C_{24}x\rfloor]\\
&~~~~~~~~~+P[E'[T'(2^{k(j-1)},2^{k(j-1)+\lfloor
C_{24}x\rfloor+1})+S_{k(j-1)+\lfloor C_{24}x\rfloor+1}]\geq
x]\\
&~~~~~~~~~+P[E'[T'(2^{kj},2^{kj+\lfloor C_{24}x\rfloor+1})+S_{kj}]\geq x]\\
&~~~~\leq C_{25}\exp(-C_{26}x) \mbox{ (note that the last two terms
above equal 0});
\end{align*}
if $x\geq k/C_{24}$,
\begin{align*}
&P[|E'[T_{k,j}(\omega')]-E'[T_{m(k(j-1),\omega),m(kj,\omega)+1}(\omega')]|\geq
x]\\
&\leq P[m(kj)\geq kj+\lfloor
C_{24}x\rfloor]+P[E'[T'(2^{k(j-1)},2^{kj+\lfloor
C_{24}x\rfloor+1})]\geq x]\leq C_{25}\exp(-C_{26}x).
\end{align*}
Then (\ref{e33}) follows immediately, which completes the proof.
\end{proof}

We are now ready to prove the limit result for $\Var [c_n]$:
\begin{proposition}\label{l5}
\begin{align*}
\lim_{n\rightarrow\infty}\frac{\Var [c_n]}{\log
n}=\frac{2}{3\sqrt{3}\pi}-\frac{1}{2\pi^2}.
\end{align*}
\end{proposition}

\begin{proof}
Applying Lemma \ref{l8}, we have
\begin{equation}\label{e34}
\left|\Var[T(0,\mathcal
{C}_{kq})]-\sum_{j=1}^{q}\Var[T_{k,j}]\right|\leq Cq\sqrt{k}+1,
\end{equation}
where $C$ is from Lemma \ref{l8}.  By the convergence of the
Ces\`{a}ro mean and (\ref{e17}), we obtain
\begin{equation}\label{e35}
\lim_{q\rightarrow\infty}\frac{\sum_{j=1}^q
\Var[T_{k,j}]}{q}=\Var[N(1/2^k)].
\end{equation}
Recall the choice $q=\lfloor\log_{2^k}n\rfloor$.  Combining
(\ref{e10}), (\ref{e37}), (\ref{e43}), (\ref{e34}) and (\ref{e35})
gives the desired result:
\begin{equation*}
\lim_{n\rightarrow\infty}\frac{\Var[c_n]}{\log
n}=\lim_{k\rightarrow\infty}\lim_{n\rightarrow\infty}\frac{\sum_{j=1}^{\lfloor\log_{2^k}
n\rfloor} \Var[T_{k,j}]}{\log
n}=\lim_{k\rightarrow\infty}\frac{\Var[N(1/2^k)]}{\log
(2^k)}=\frac{2}{3\sqrt{3}\pi}-\frac{1}{2\pi^2}.
\end{equation*}
\end{proof}

\subsection{Proofs of Theorem \ref{t1} and Corollary
\ref{c1}}\label{s34}

\begin{proof}[Proof of Theorem \ref{t1}]
The proof of Theorem 1.1 in \cite{20}, together with Proposition
\ref{l4}, implies (\ref{t14}) and (\ref{t13}) immediately.

By Lemma \ref{l10}, (\ref{e38}) and (\ref{e27}), there exist
$C_1,C_2>0$, such that for all $2^q\leq n<2^{q+1}$, $x\geq 4K$,
where $K$ is from Lemma \ref{l10}, we have
\begin{align*}
&P[b_{0,n}-c_n\geq x]\\
&\leq P[T(0,\mathcal {C}_{q+1})-c_{2^q}\geq x]\\
&\leq P[m(q+1)\geq q+\lfloor x/(2K)\rfloor]+P[T'(2^{q},2^{q+\lfloor
x/(2K)\rfloor})\geq
x/2]+P[S_{q}\geq x/2]\\
&\leq C_1\exp(-C_2x).
\end{align*}
Then there is a universal $C_3>0$, such that
\begin{equation}\label{e39}
E[|b_{0,n}-c_n|]\leq C_3,~~E[|b_{0,n}-c_n|^2]\leq C_3.
\end{equation}
Then (\ref{t12}) follows from (\ref{e14}) and (\ref{e39}).

By (\ref{e43}) and (\ref{e39}), one easily obtains that there exists
$C_4>0$ such that
\begin{equation*}
|\Var[b_{0,n}]-\Var[c_n]|\leq C_4\sqrt{\log n}.
\end{equation*}
Then (\ref{t22}) follows from this and Proposition \ref{l5}.

For $z\in A(2^q,2^{q+1})$ and $q\geq 2$, define
\begin{equation*}
Y(z):=T(0,\partial B(2^{q-1}))+T(z,\partial B(z,2^{q-1})).
\end{equation*}
It is obvious that $T(0,z)\geq Y(z)$ and
\begin{equation*}
T(0,z)-Y(z)\leq |T(0,\mathcal {C}_{q+2})-T(0,\partial
B(2^{q-1}))|+|T(z,\mathcal {C}_{q+2})-T(z,\partial B(z,2^{q-1}))|.
\end{equation*}
Using this inequality, similarly to the above argument, one can show
that there is a universal $C_5>0$ such that
\begin{equation}\label{e44}
E[|T(0,z)-Y(z)|]\leq C_5,~~E[|T(0,z)-Y(z)|^2]\leq C_5,
\end{equation}
see the proofs of Corollary 5.11 and 5.12 in \cite{21}.  Combining
(\ref{e14}) and (\ref{e44}) gives (\ref{t11}) immediately.  By
(\ref{e43}) and (\ref{e44}) and using the triangle inequality for
the norm $\|\cdot\|_2$, there exist universal $C_6>0$ such that
\begin{equation*}
|\Var[T(0,z)]-2\Var[c_{2^{q-1}}]|\leq C_6\sqrt{q}.
\end{equation*}
Then we derive (\ref{t21}) from this and Proposition \ref{l5},
finishing the proof.
\end{proof}

\begin{proof}[Proof of Corollary \ref{c1}]
Using the Theorem in \cite{10}, one can derive the following result
easily: As $n\rightarrow\infty$,
\begin{align}
&\frac{b_{0,n}-E[b_{0,n}]}{\sqrt{\Var
[b_{0,n}]}}\stackrel{d}\longrightarrow N(0,1),\label{e5}\\
&\frac{T(0,nu)-E[T(0,nu)]}{\sqrt{\Var
[T(0,nu)]}}\stackrel{d}\longrightarrow N(0,1).\label{e4}
\end{align}
Note that (\ref{e4}) is a particular case of Corollary 5.13 in
\cite{21}, and the proof of (\ref{e5}) is similar to that of
(\ref{e4}).  Corollary \ref{c1} follows from (\ref{e5}), (\ref{e4})
and Theorem \ref{t1} immediately.
\end{proof}

\section{A discussion on the limit shape}\label{s4}
Since there is a ``shape theorem" for general non-critical FPP (see
e.g. \cite{26}), it is natural to ask if there is an analogous
result for critical FPP. For that purpose, we view the sites of
$\mathbb{V}$ as hexagons, and define $W(n):=\{v\in
\mathbb{V}:T(0,v)\leq n\}$ for $n\in \mathbb{N}$.  Unlike the
non-critical case, there are ``large holes" in $W(n)$, and the
geometry of $W(n)$ is not easy to analyze.  Instead, it would be
easier to deal with the set $\overline{W}(n)$ obtained from $W(n)$
by filling in the holes. That is, $\overline{W}(n)$ is the closed
set surrounded by the outer boundary of $W(n)$.  We shall discuss
the limit boundary of $\overline{W}(n)$ below.

Similarly to the proof of Proposition \ref{l1}, we inductively
define the $n$-th innermost disjoint yellow circuit $\mathcal {C}_n$
surrounding 0.  For simplicity, if the hexagon centered at 0 is
yellow, we also call it the first innermost circuit surrounding 0.
Let $\mathcal {C}_n^*:=\mathcal {C}_n\cup$ blue clusters which touch
$\mathcal {C}_n$.  Observe that $\partial\overline{W}(n)$ is equal
to the outer boundary of $\mathcal {C}_n^*$.  By a color switching
trick one obtains that $\partial\overline{W}(n)$ is equal in
distribution to the $n$-th innermost cluster boundary loop $\mathcal
{L}_n$ surrounding 0. Thus, it is expected that as
$n\rightarrow\infty$, the limit boundary of $\overline{W}(n)$ under
appropriate scaling has the same distribution as a ``typical'' loop
of full-plane CLE$_6$.  Note that $\overline{W}(n)$ grows
approximately exponentially as time passes.

In \cite{31}, Kemppainen and Werner proved the invariance of
full-plane CLE$_{\kappa}$ under the inversion $z\mapsto 1/z$ for
$0<\kappa\leq 4$.  Although the analogous result has not been proved
for the case $4<\kappa<8$, Miller and Sheffield \cite{22} proved the
time-reversal symmetry of whole-plane SLE$_{\kappa}$ for $\kappa\in
(0,8]$.  To establish the following statement, for a loop $\mathcal
{L}$ that surrounds 0, we let $R_{in}(\mathcal
{L}):=\dist(0,\mathcal {L})$ denote the inner radius of $\mathcal
{L}$ and let $R_{out}(\mathcal
{L}):=\inf\{r:\overline{\mathbb{D}}_r\supset \mathcal {L}\}$ denote
the outer radius of $\mathcal {L}$.  By the argument above, it is
expected that the following statement holds: As
$n\rightarrow\infty$, the distributions of
$\partial\overline{W}(n)/R_{in}(\partial\overline{W}(n))$ and
$\partial\overline{W}(n)/R_{out}(\partial\overline{W}(n))$ converge,
and we denote the two limit distributions by $\mu_{in}$ and
$\mu_{out}$, respectively. Furthermore, $\mu_{out}$ is the image of
$\mu_{in}$ under $z\mapsto 1/z$.

\end{document}